\documentclass[a4paper, 12pt,
 reqno]{amsart}
\usepackage{latexsym,amscd,amssymb,amsmath,amsthm,comment,url}
\usepackage[dvips]{graphicx}
\usepackage[all]{xy}
\xyoption{poly}
\xyoption{knot}
\makeatletter
\newcounter{cntx}
\newcounter{cnty}
\newcounter{endx}
\newcounter{endy}
\def\blackcirclefill<#1>(#2,#3)(#4,#5){
  \setcounter{cnty}{#5}
  \setcounter{endx}{#2}
  \addtocounter{endx}{1}
  \setcounter{endy}{#3}
  \addtocounter{endy}{1}
  \loop
    \setcounter{cntx}{#4}
    {%
      \loop
        \put(\thecntx,\thecnty){\circle*{#1}}  
        \stepcounter{cntx}  
        \ifnum\thecntx < \theendx \repeat
    }%
    \stepcounter{cnty}
  \ifnum\thecnty < \theendy \repeat
}

\def\whitecirclefill<#1>(#2,#3)(#4,#5){
  \setcounter{cnty}{#5}
  \setcounter{endx}{#2}
  \addtocounter{endx}{1}
  \setcounter{endy}{#3}
  \addtocounter{endy}{1}
  \loop
    \setcounter{cntx}{#4}
    {%
      \loop
        \put(\thecntx,\thecnty){\circle{#1}}  
        \stepcounter{cntx}  
        \ifnum\thecntx < \theendx \repeat
    }%
    \stepcounter{cnty}
  \ifnum\thecnty < \theendy \repeat
}

\theoremstyle{plain}
  \newtheorem{thm}{Theorem}[section]
  \newtheorem{prop}[thm]{Proposition}
  \newtheorem{lem}[thm]{Lemma}
  \newtheorem{cor}[thm]{Corollary}
\theoremstyle{definition}
  \newtheorem{dfn}[thm]{Definition}
  \newtheorem{exmp}[thm]{Example}
\theoremstyle{remark}
  \newtheorem{rem}[thm]{Remark}
\let\opn\operatorname %abbreviation of \operatorname
\let\term\emph
\numberwithin{equation}{section}
\def\NN{\mathbb{N}} %the monoid of natural numbers with 0
\def\ZZ{\mathbb{Z}} %the ring of integers
\def\RR{\mathbb{R}} %the field of real numbers
\def\kk{\Bbbk} %base field
\def\m{\mathfrak{m}} %maximal ideal
\def\p{\mathfrak{p}} %prime ideal
\def\q{\mathfrak{q}}

\def\rH{\widetilde{H}} %reduced homology
\let\e\varepsilon %incidence function
\let\s\sigma %cell
\let\C\Sigma %conical complex
\let\t\tau %cell
\let\u\upsilon %cell
\let\i\iota %injection
\def\LL{\mathcal L}
\def\MM{\mathcal M} %monoidal complex
\def\M{\mathbf M} %affine monoid
\def\bL{\mathbf L} % lattice points
\def\bE{\mathbf E} 
\def\bF{\mathbf F}
\def\cC{\mathcal C} 
\def\Cec{\check{C}}
\def\for{\mathbf U}
\def\longto{\longrightarrow} %\longto=\longrightarrow
\def\bra#1{[#1]} % square bracket
\def\mbra#1{\{ #1\}} % middle bracket
\def\set#1#2{\mbra{\,#1\mid #2\,}} %set
\def\sM{|\MM |}

\def\supp{\opn{supp}} %support
\def\cell{\mathcal X} %cell complex
\def\op{\mathsf{op}} %opposite
\def\sph{\mathbb S} %sphere
\def\idmap{\opn{Id}} %\identity map
\def\fring#1{\kk \bra{#1}} %face ring
\def\Mod{\opn{Mod}}    
\def\*Mod{{}^*\!\opn{Mod}}   

\def\Sq{\opn{Sq}}      
\def\InjSq{\opn{Inj-Sq}}
\def\Pic{\opn{Pic}}
\def\EE{{}^*\! E}   
\def\Mgr{\Mod_{\MM}} %the category of \ZM-graded modules
\def\Mbgr{\Mod_{\bMM}}
\def\Lgr{\Mod_{\LL}}
\def\Lvect{\opn{Vect}_{\LL}}

\let\colimit\varinjlim %colimit (direct limit)
\def\Hom{\opn{Hom}}    
\def\Ext{\opn{Ext}}    
\def\inHom{\Hom^\bullet}
\def\int{\opn{int}}   

\def\Db{{\mathsf D}^b} %bounded derived category
\def\Cb{{\mathsf K}^b} %homotopy category

\def\upl{{}^+\!}
\def\cpx#1{#1^{\bullet}} %cochain complexes
 
\def\<{{\langle}}
\def\>{{\rangle}}

\def\DC{D^\bullet}

\def\bM{\overline{\M}}
\def\tM{\widetilde{\M}}
\def\bMM{\widetilde{\MM}}
\def\oR{\widetilde{R}}

\def\bR{\mathbf R}
\def\baR{\overline{R}}
\def\ols{\widetilde{\s}}
\def\too{\longrightarrow}
\def\chara{\operatorname{char}}
\makeatother
\title[Dualizing complexes of seminormal toric face rings]
{Dualizing complexes of seminormal affine semigroup rings and toric face rings}
\author{Kohji Yanagawa}
\thanks{Partially supported by Grant-in-Aid for Scientific Research (c) (no.19540028).}
\address{Department of Mathematics, Kansai University,
Suita 564-8680, Japan}
\email{yanagawa@ipcku.kansai-u.ac.jp}
\begin{document}
%

%---title---
%
\maketitle

\begin{abstract}
We characterize the {\it seminormality} of an affine semigroup ring  
in terms of the dualizing complex, and the {\it normality} of a Cohen-Macaulay semigroup ring  by the ``shape" of the canonical module. 
%This characterization does not use the grading of $\omega_R$.  
We also characterize the seminormality of a toric face ring  
in terms of the dualizing complex. 
A toric face ring is a simultaneous generalization of Stanley-Reisner rings and affine semigroups. 
%For rings of this class, the seminormality is more essential than the normality as shown by Nguyen. 
\end{abstract}

\section{Introduction}
Let $\M$ be a finitely generated additive submonoid of $\ZZ^d$ (i.e., $\M$ is an affine semigroup) with $\ZZ \M \cong \ZZ^d$, 
and $\cC(\M):=\RR_{\geq 0} \M \subset \ZZ^d \otimes_\ZZ \RR \cong \RR^d$ the polyhedral cone spanned by $\M$. Set $\bM := \ZZ \M \cap \cC(\M)$. 
Throughout the paper, we assume that $\M$ is {\it positive}, that is, $\M$ has no invertible element except 0. 

In the former half of the present paper, we study the affine semigroup ring  $\kk[\M]=\bigoplus_{a \in \M} \kk \, 
	x^a$ of $\M$ over a field $\kk$. Now we have $\dim \kk[\M]=d$. 
It is a classical result by Hochster, Stanley and Danilov that if $R=\kk[\M]$ is normal (equivalently,  $\M = \bM$), 
then $R$ is Cohen-Macaulay and the canonical module $\omega_R$ has an easy description (cf.  \cite[Theorem~6.3.5]{BH}).  
On the other hand, the behavior of non-normal affine semigroup rings is delicate and complicated, and many works have been done on this subject.  
%One of the hopeful directions is to consider special classes.  

\begin{dfn}
Let $A$ be a reduced noetherian commutative ring, and $Q(A)$ its total quotient ring. 
We say $A$ is {\it seminormal},  if $a \in Q(A)$ and $a^2, a^3 \in A$ 
imply $a \in A$. 
\end{dfn}

This notion is much more natural than it seems. 
In fact, it is known that $R$ is seminormal if and only if  $\Pic R \cong \Pic (R[x]) $. 
See \cite{Sw} and the references cited therein. 

The seminormality of  an affine semigroup ring $R=\kk[\M]$ is characterized in 
a combinatorial (resp. homological) way by Reid and Roberts \cite{RR} (resp. 
Bruns, Li and  R\"omer \cite{BLR}).    In the present paper, we will give a new characterization 
using the dualizing complex. 
Our characterization is relatively closer to that in \cite{BLR}.  
However,  contrary to their result, ours does not use the $\ZZ^d$-grading of the local cohomology modules 
(or the dualizing complex).  To introduce our result, we need some preparation. 

For a face $F$ of the cone $\cC(\M)$, 
$\M_F := \M \cap F$ is a submonoid of $\M$. 
The semigroup ring $\kk[\M_F]$ can be seen as a quotient ring 
of  $R$, and  its normalization $\kk[\bM_F]$  
has the natural $R$-module structure. 
Then we have the following complex. 
$$\upl \cpx I_R : 0 \too \upl I^{-d}_R\too \upl I^{-d+1}_R \too \cdots \too \upl I^0_R \too 0,$$
$$\upl I^{-i}_R=\bigoplus_{\substack{ F: \, \text{a face of} \, \cC(\M) \\ \dim F =i}} 
\kk[\bM_F].$$
The differential map $\partial: \upl I^{-i}_R \to \upl I^{-i+1}_R$ is 
the combination of the natural surjections $\kk[\bM_F] \twoheadrightarrow \kk[\bM_G]$ 
for faces $F,G$ with $F \supset G$ and $\dim F=\dim G+1$.  

\medskip

\noindent{\bf Proposition~\ref{+I_R}.} 
{\it For a semigroup ring $R=\kk[\M]$, it  is seminormal if and only if $\upl \cpx I_R$ 
is quasi-isomorphic to the dualizing complex $\cpx D_R$.  }

\medskip

We can characterize the normality of $\kk[\M]$ using the dualizing complex in a similar way. 
As a byproduct of this observation, we have the following (unexpected) result.

\medskip

\noindent{\bf Theorem~\ref{normal}}
{\it For $R=\kk[\M]$, the following are equivalent.
\begin{itemize}
\item[(a)] $R$ is normal. 
\item[(b)] $R$ is Cohen-Macaulay and the canonical module $\omega_R$ is isomorphic to 
the ideal $( \, \, x^a \mid a \in \M \cap \int (\cC(\M)) \, )$ of $R$ 
as (graded or nongraded) $R$-modules.
\end{itemize}
}

%We can consider the subcomplex $\cpx I_R$ of $\upl \cpx I_R$ with 
%$$I^{-i}_R=\bigoplus_{\substack{ F: \, \text{a face of} \, \cC(\M) \\ \dim F =i}} \kk[\M_F]$$
%for each $0 \le i \le d$. If $R$ is normal then $\cpx I_R = \upl \cpx I_R$. 
%Hence, if $R$ is normal, then $\cpx I_R$ is quasi-isomorphic to the dualizing complex. 
%Moreover, we have the following.  

\medskip

The implication (a) $\Rightarrow$ (b) is a classical result (see above). 

%\noindent{\bf Theorem~\ref{normal}}
%{\it For $R=\kk[\M]$, the following are equivalent.
%\begin{itemize}
%\item[(i)] $R$ is normal. 
%\item[(ii)] The complex $\cpx I_R$ is quasi-isomorphic to the dualizing complex $\cpx D_R$. 
%\item[(iii)] $R$ is Cohen-Macaulay and the canonical module $\omega_R$ is isomorphic to 
%the ideal $( \, \, x^a \mid a \in \M \cap \int (\cC(\M)) \, )$ of $R$ 
%as (graded or nongraded) $R$-modules.
%\end{itemize}
%}

Stanley-Reisner rings and affine semigroup rings are important subjects of 
combinatorial commutative algebra. 
The notion of {\it toric face rings}, which originated in an earlier work of 
Stanley \cite{St}, generalizes both of them, and has been 
studied by Bruns, R\"omer, and their coauthors (e.g. \cite{BG,BKR, IR}). 
%Contrary to these classical examples,  
%a toric face ring does not admit a nice multi-grading in its most general setting. 
Roughly speaking, to make a toric face ring $\kk[\MM]$ from a (locally) polyhedral CW complex $\cell$, 
we  assign each cell $\s \in \cell$ 
an affine semigroup $\M_\s \subset \ZZ^{\dim \s +1}$, 
%so that certain compatibility is satisfied, 
and ``glue" their semigroup rings $\kk[\M_\s]$ along with $\cell$. 

Recently, Nguyen \cite{Ngu} studied seminormal toric face rings mainly focusing on  
the local cohomology modules, but he also remarked that $\kk[\MM]$ is seminormal if and only if 
$\kk[\M_\s]$ is seminormal for all $\s$. In this sense, the seminormality is a natural condition  
for toric face rings. 

Generalizing the construction for affine semigroup rings, a toric face ring 
$\kk[\MM]$ of dimension $d$ admits the cochain complex $\upl \cpx I_R$ of the form 
$$
0 \longto \upl I^{-d}_R \longto \upl I^{-d + 1}_R 
\longto \cdots \longto \upl I^0_R \longto 0
$$
with 
$$\upl I^{-i}_R := \bigoplus_{\substack{\s \in \cell \\ \dim \s = i-1}} \kk[\bM_\s],$$
where $\kk[\bM_\s]$ is the normalization of $\kk[\M_\s]$. 

\medskip

\noindent{\bf Theorem~\ref{main}} 
{\it If a toric face ring $R=\kk[\MM]$ is seminormal, then $\upl \cpx I_R$ is quasi-isomorphic to the dualizing complex $\cpx D_R$. 
(The converse is also true. See Proposition~\ref{conv. of main})} 

\medskip

Under the assumption that each $\kk[\MM_\s]$ is normal 
(of course, $\upl I^{-i}_R= \bigoplus_{\dim \s = i-1} \kk[\M_\s],$ 
in this case), the above theorem was proved by the present author and Okazaki 
(\cite[Theorem~5.2]{OY}). Even in this case, the proof requires quite technical argument, 
since $R$ is not a graded ring in the usual sense. 
The proof of Theorem~\ref{main} heavily depends on \cite[Theorem~5.2]{OY},  
but we have to  make more effort.   

Finally, for an arbitrary toric face ring $R=\kk[\MM]$, we study the local cohomology modules $H_\m^i(R)$ 
at the ``graded" maximal ideal $\m$. Let $\upl R$ (resp. $\oR$) be the seminormalization (resp. cone-wise normalization) 
of $R$. Both of them are toric face rings supported by the same CW complex $\cell$ as $R$, 
but the construction of the latter is  not straightforward (see Example~\ref{bMM}). 
In \S6, we show that $H_\m^i(\upl R) \subset H_\m^i(R)$, and 
$H_\m^i(\oR)\ne 0$ implies $H_\m^i(R)\ne 0$. 
Hence we have; 
$$\text{$R$ is Cohen-Macaulay} \Longrightarrow \text{$\upl R$ is Cohen-Macaulay} 
\Longrightarrow \text{$\oR$ is Cohen-Macaulay}.$$
We remark that the Cohen-Macaulay property of $\oR$ only depends on the topology of the underlying space of $\cell$ 
(and $\chara(\kk)$).

\medskip

\noindent{\bf Convention.} In this paper, we use the following notation: 
For a commutative ring $A$, $\Mod A$ denotes the category of $A$-modules. 

For cochain complexes $\cpx M$ and $\cpx N$, 
$\cpx M \cong \cpx N$ means that two complexes are isomorphic in the derived category, and 
$\cpx M =\cpx N$ means that these are isomorphic as (explicit) complexes.  
If $\cpx M \cong \cpx N$, we say these two complexes are {\it quasi-isomorphic} 
(especially when  a direct quasi-isomorphism $\cpx M \to \cpx N$ or $\cpx N \to \cpx M$ exists).

While the word ``dualizing complex" sometimes means its isomorphism class in the derived category,  
we use the convention that the dualizing complex 
$\cpx D_A$ of a noetherian ring $A$ is the one of the form 
$$0 \too D_A^{-\dim A} \too \cdots \too D_A^{-1} \too D_A^0 \too 0$$
with 
\begin{equation}\label{normal form}
D_A^{-i} = \bigoplus_{\substack{\p \in \operatorname{Spec} A \\ \dim A/\p =i}} E(A/\p), 
\end{equation}
where $E(A/\p)$ is the injective envelope of $A/\p$. 

In this paper, we freely use the $\ZZ^d$-graded versions of Matlis duality and local duality. 
These are implicit in Chapters 5 and 6 of \cite{BH}, but the detailed argument is found in \cite{GW}. 

\section{Dualizing complexes of seminormal affine semigroup rings}
For the convention and notation about an affine semigroup $\M \subset \ZZ^d$ and the cone 
$\cC(\M) \subset \RR^d$ spanned by $\M$, see the end of the  previous section. 

Let 
$$\kk[\M]:= \bigoplus_{a \in \M} \kk \, x^a \subset \kk[x_1^{\pm 1}, \ldots, x_d^{\pm 1}]$$
be the semigroup ring of $\M$ over a field $\kk$. Here, for $a =(a_1, \ldots, a_d) \in \ZZ^d$, 
$x^a$ denotes the monomial $\prod_{i=1}^d x_i^{a_i}.$  
Clearly, $R:=\kk[\M]$ is a $\ZZ^d$-graded ring, and $\*Mod R$ denotes 
the category of $\ZZ^d$-graded $R$-modules. 

For $M =\bigoplus_{a \in \ZZ^d} M_a \in \*Mod R$, set 
$$M_{\cC(\M)}:= \bigoplus_{a \in \ZZ^d \cap \cC(\M)} M_a.$$
It is clear that $M_{\cC(\M)}$ is a $\ZZ^d$-graded $R$-submodule of $M$, 
and we call it the {\it $\cC(\M)$-graded part} of $M$. 
Similarly, for a cochain complex $\cpx M$ in $\*Mod R$, we can defined a subcomplex $(\cpx M)_{\cC(\M)}$. 

For a face $F$ of $\cC(\M)$, 
$$\M_F := \M \cap F$$
is a submonoid of $\M$. Consider the monomial  ideal (i.e., $\ZZ^d$-graded ideal)
$$\p_F := ( x^a \mid a \in \M \setminus \M_F)$$
of $R$. Since $R/\p_F$ is isomorphic to the affine semigroup ring $\kk[\M_F]$ of $\M_F$, $\p_F$ is a prime ideal. 
Conversely, any monomial prime ideal coincide with $\p_F$ for some $F$. 
We regard $\kk[\M_F]$ as an $R$-module through $R/\p_F \cong \kk[\M_F]$. 

For a face $F$ of $\cC(\M)$, $T_F := \set{x^a}{a \in \M_F} \subset R$ is a multiplicatively 
closed subset. So we have the localization $T_F^{-1} R$ of $R$ by $T_F$.
The {\it C\v ech complex} $\cpx \Cec_R$ is defined as follows: 
$$
\cpx \Cec_R: 0 \longto \Cec^0_R \stackrel{\partial}{\longto} \Cec^1_R \stackrel{\partial}{\longto} 
\cdots \stackrel{\partial}{\longto} \Cec^d_R \longto 0, 
$$ 
where 
$$
\Cec^i_R := \bigoplus_{\substack{ F: \, \text{a face of} \, \cC(\M) \\ \dim F =i}} T_F^{-1}R. 
$$
The differential map $\partial:\Cec^i_R \to \Cec^{i+1}_R$ is given by
$$
\partial(x) = \sum_{\substack{G \supset F \\ \dim G = i+1}} \e(G,F) \cdot \i_{G, F}(x), 
$$
where  $\i_{G, F}$ is the natural injection $T^{-1}_F R \too T^{-1}_G R$
for $G \supset F$, and $\e(G,F)$ is the incidence function of the regular CW complex given by a 
cross section of $\cC(\M)$. The precise information on  $\e(G,F)$ is found in \cite[\S 6.2]{BH}, and 
we will use this function later in a more general situation. 
Here we just remark that $\e(G,F)=\pm 1$ for all $F,G$ with $G \supset F$ and $\dim G = \dim F+1$, 
and this signature makes $\cpx \Cec_R$  a cochain complex.    

As shown in \cite[Theorem~6.2.5]{BH}, the local cohomology module $H_\m^i(R)$ at the graded maximal ideal $\m:=(x^a \mid 0 \ne a \in \M)$ 
is isomorphic to $H^i(\cpx \Cec_R)$ in $\*Mod R$.   Moreover, $\cpx \Cec_R$ is a ($\ZZ^d$-graded) flat resolution of $\bR \Gamma_\m R$.   

\medskip
 
The $\ZZ^d$-graded Matlis dual $(T_F^{-1}R)^\vee$ of $T_F^{-1}R$ is of the form 
$$(T_F^{-1}R)^\vee =\bigoplus_{a \in \M_F -\M} \kk \, e_a,$$
where $e_a$ is a basis element with the degree $a$, and 
$$\M_F -\M =\{ \, b-c \mid b \in \M_F \ \text{and} \ c \in \M \, \}.$$  
The multiplication map $x^a \times (-) \,: [(T_F^{-1}R)^\vee]_b \too [(T_F^{-1}R)^\vee]_{a+b}$ is surjective for all 
$a \in \M$ and $b \in \ZZ^d$. 
By the flatness of $T_F^{-1}R$ and \cite[Lemma~11.16]{MS}, $(T_F^{-1}R)^\vee$ is an injective object in $\*Mod R$,  
moreover, it is the injective envelope $\EE (\kk[\M_F])$ of 
$\kk[\M_F]=R/\p_F$ in $\*Mod R$.

The $\ZZ^d$-graded Matlis dual $\cpx J_R:= (\cpx \Cec_R)^\vee$ of $\cpx \Cec_R$ is of the form
$$\cpx J_R : 0 \too J^{-d}_R\too J^{-d+1}_R \too \cdots \too J^0_R \too 0,$$
$$J^{-i}_R=\bigoplus_{\substack{ F: \, \text{a face of} \, \cC(\M) \\ \dim F =i}} 
\EE (\kk[\M_F]).$$
The differential map $\partial: J^{-i}_R \to J^{-i+1}_R$ is given by
$$
\partial(x) = \sum_{\substack{G \subset F \\ \dim G = i-1}} \e(F,G) \cdot p_{G, F}(x) 
$$
for $x \in \EE[\M_F] \subset J_R^{-i}$. Here $p_{G,F}: \EE(\kk[\M_F]) \to \EE(\kk[\M_G])$ 
is the Matlis dual of $\i_{F,G}$, and also induced by the map $\kk[\M_F] \to \EE(\kk[\M_G])$
which is the composition of the natural surjection $\kk[\M_F] \twoheadrightarrow \kk[\M_G]$ and  
the inclusion $\kk[\M_G] \hookrightarrow \EE(\kk[\M_G])$. 

As is well-known,  $\cpx J_R$ is quasi-isomorphic to the dualizing complex $\cpx D_R$ of $R$, 
moreover, it is nothing other than the dualizing complex of $R$ in the $\ZZ^d$-graded context 
(see \cite[Proposition~4.4]{SS} , also \cite{I}).

For a face $F$ of the polyhedral cone $\cC(\M)$, we regard  
$$\kk[\ZZ\M_F \cap F]:=\bigoplus_{b \in \ZZ \M_F \cap F} \kk \, x^b$$
as a $\ZZ^d$-graded $R$-module by 
$$
x^a x^b=\begin{cases}
x^{a+b} & \text{if $a \in \M_F$,}\\
0 & \text{otherwise,}
\end{cases}
$$
for $x^a \in R=\kk[\M]$ and $x^b \in \kk[\ZZ\M_F \cap F]$. 
Note that $\kk[\ZZ\M_F \cap F]$ is the normalization of $\kk[\M_F]$, and 
%The $\kk[\M]$-module structure of $\kk[\ZZ \M_F \cap F]$ defined above coincides with the one given by the composition 
%of the natural surjection $\kk[\M] \twoheadrightarrow \kk[\M_F]$ and the normalization $ \kk[\M_F] \hookrightarrow \kk[\ZZ \M_F \cap F]$.  
$$\EE(\kk[\M_F])_{\cC(\M)} \cong \kk[\ZZ \M_F \cap F]$$ 
as $R$-modules. Let $F,G$ be faces of  $\cC(\M)$ with $F \supset G$. 
As $R$-modules,  $\kk[\ZZ \M_G \cap G]$ is a quotient module of  $\kk[\ZZ \M_F \cap F]$ 
(note that  $\ZZ \M_G$ is a sublattice of $\ZZ \M_F \cap G$). 
Hence there is the $\ZZ^d$-graded surjection $\pi_{G,F} : \kk[\ZZ \M_F \cap F] \too \kk[\ZZ \M_G \cap G]$, which is the $\cC(\M)$-graded part of $p_{G,F}$ 
(if $\dim G =\dim F-1$).

Hence the $\cC(\M)$-graded part $$\upl \cpx I_R:= (\cpx J_R)_{\cC(\M)}$$  
of the complex $\cpx J_R$ is of the form   
$$\upl \cpx I_R : 0 \too \upl I^{-d}_R\too \upl I^{-d+1}_R \too \cdots \too \upl I^0_R \too 0,$$
$$\upl I^{-i}_R=\bigoplus_{\substack{ F: \, \text{a face of} \, \cC(\M) \\ \dim F =i}} 
\kk[\ZZ\M_F \cap F].$$
The differential map $\partial: \upl I^{-i}_R \to \upl I^{-i+1}_R$ is given by
$$
\partial(x) = \sum_{\substack{G \subset F \\ \dim G = i-1}} \e(F,G) \cdot \pi_{G, F}(x), 
$$
for $x \in \kk[\M_F] \subset \upl I_R^{-i}$. 

\medskip

As is well-known, $R=\kk[\M]$ is normal if and only if $\M=\bM:=\ZZ\M \cap \cC(\M)$. 
We can characterize the seminormality of $R$ in a similar way. 
For a face $F$ of $\cC(\M)$, $\int (F)$ denotes its relative interior.
Clearly,  $$\cC(\M)=\bigsqcup _{F :\, \text{a face of} \, \cC(\M)} \int (F).$$ 
Set 
\begin{equation}\label{monoid seminormalization}
\upl \M := \bigsqcup_{F :\, \text{a face of} \, \cC(\M)} \ZZ \M_F \cap \int (F).
\end{equation}
Then $\upl \M$ is an affine semigroup with $\M \subseteq \upl \M \subseteq \bM$ and $\upl (\upl \M) =\upl \M$.

\begin{thm}[L. Reid and L.G. Roberts \cite{RR}, Bruns, Li and R\"omer \cite{BLR}]\label{RRBLR}
For an affine semigroup ring $R=\kk[\M]$, the following are equivalent.      
\begin{itemize}
\item[(i)]$R$ is seminormal. 
\item[(ii)] $\M =\upl \M$.
\item[(iii)]  $H_\m^i(R)_a \ne 0$ for $a \in \ZZ^d$ implies $-a \in \cC(\M)$.
\end{itemize}
Hence $\upl R :=\kk[\upl \M]$ is the seminormalization of $R=\kk[\M]$. 
\end{thm}

In the above theorem, the equivalence between (i) and (ii) (resp. (i) and (iii)) is 
\cite[Theorem~4.3]{RR} (resp. \cite[Theorem~4.7]{BLR}).

\begin{exmp}\label{seminormal, but nonnormal}
For the additive submonoid 
$$\M=\{\, (m,n) \mid m \ge 0, n \ge 1\, \} \cup 
\{\, (2m,0) \mid m \ge 0 \, \} $$
 of $\NN^2$, $\kk[\M]$ is seminormal, but not normal. 
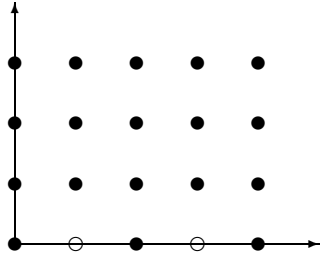
\begin{center}
\begingroup  % \setlength
\setlength\unitlength{8mm} 
\begin{picture}(7,6)(0,-1)  
\blackcirclefill<.2>(4,3)(0,1) 
\put(0,0){\circle*{.2}}  
\put(1,0){\circle{.2}}  
\put(2,0){\circle*{.2}} 
\put(3,0){\circle{.2}}
\put(4,0){\circle*{.2}} 
%\put(5,0){\circle{.2}}
%\put(6,0){\circle*{.2}} 
%\blackcirclefill<.2>(6,0)(4,0)  % \blackcirclefill 
\put(0,0){\vector(1,0){5}}  % 
\put(0,0){\vector(0,1){4}}  % 
\end{picture}
\endgroup
\end{center}
\end{exmp}

\begin{prop}\label{+I_R}
If  $R=\kk[\M]$ is seminormal,  then $\upl \cpx I_R$ is isomorphic to the $\ZZ^d$-graded dualizing 
complex $\cpx J_R$  in the derived category $\Db(\*Mod R)$,  
hence  $\upl \cpx I_R \cong \cpx D_R$ in  $\Db(\Mod R)$. 
Conversely, if $\upl \cpx I_R \cong \cpx D_R$ in  $\Db(\Mod R)$ then $R$ is seminormal. 
\end{prop}

\begin{proof}
We start from the proof of the first assertion. 
Since $H_\m^i(R)^\vee \cong H^{-i}(\cpx J_R)$ by the local duality theorem, 
$H^i(\cpx J_R)_a \ne 0$ implies $a \in \cC(\M)$ by Theorem~\ref{RRBLR}. 
Hence the $\cC(\M)$-graded part $\upl \cpx I_R$ of $\cpx J_R$ 
is quasi-isomorphic to $\cpx J_R$ itself.   

Next, we show the last assertion. 
For the seminormalization $\upl R$ of $R$, the explicit computation gives the isomorphism  
$\upl \cpx I_R = \upl \cpx I_{\upl R}$ as cochain complexes of $R$-modules. 
We just shown that $\upl \cpx I_{\upl R} \cong \cpx D_{\upl R}$ in $\Db(\Mod \upl R)$.  
Hence $\upl \cpx I_{\upl R} \cong \cpx D_{\upl R}$  also in $\Db(\Mod R)$. Since $\upl R$ is a finitely generated $R$-module, 
$\Hom^\bullet_R( \upl \cpx I_{\upl R}, \cpx D_R) \cong \upl R$ in $\Db(\Mod R)$. 
Clearly, we also have $\Hom^\bullet_R( \upl \cpx I_R, \cpx D_R) \cong R$. 
So taking the functor $\Hom^\bullet_R(-, \cpx D_R)$ to $\upl \cpx I_R = \upl \cpx I_{\upl R}$, 
we have $R \cong \upl R$ as $R$-modules. 
It means that $R=\upl R$, and hence $R$ is seminormal. 
\end{proof}

\section{The normality and the canonical module of an affine semigroup ring}
Consider the following subcomplex  of $\upl \cpx I_R$: 
$$\cpx  I_R : 0 \too I^{-d}_R\too I^{-d+1}_R \too \cdots \too I^0_R \too 0,$$
$$ I^{-i}_R=\bigoplus_{\substack{ F: \, \text{a face of} \, \cC(\M) \\ \dim F =i}} 
\kk[\M_F].$$

If $R$ is normal, then $\kk[\M_F]$ is normal for all $F$ and $\cpx I_R = \upl \cpx I_R$. Hence, in this case, $\cpx I_R$ is quasi-isomorphic to 
the dualizing complex $\cpx D_R$. This is a well-known result 
essentially appears in \cite[\S6.3]{BH}. 
The next result states that the converse also holds.

\begin{thm}\label{normal}
For an affine semigroup ring $R=\kk[\M]$, the following are equivalent.
\begin{itemize}
\item[(i)] $R$ is normal. 
\item[(ii)] The complex $\cpx I_R$ is quasi-isomorphic to the dualizing complex $\cpx D_R$. 
\item[(iii)] $R$ is Cohen-Macaulay and the canonical module $\omega_R$ is isomorphic to 
the ideal $W_R := ( \, \, x^a \mid a \in \M \cap \int (\cC(\M)) \, )$ of $R$ 
in $\Mod R$. 
\end{itemize}
\end{thm}

The implication (i) $\Rightarrow$ (iii)  is a classical result due to Hochster, Stanley and Danilov.  
Note that if $R$ is normal then $\omega_R \cong W_R$ even in $\*Mod R$.

\begin{proof}
(i) $\Rightarrow$ (ii): We have mentioned above. 

(ii)  $\Rightarrow$ (iii): The assertion follows form  direct computation similar to  the proof of \cite[Theorem~6.3.4]{BH} 
(but we have to take the $\ZZ^d$-graded Matlis dual). 

(iii) $\Rightarrow$ (i): 
Since $W_R$ and $\omega_R$ are $\ZZ^d$-graded modules, $\Hom_R(W_R, \omega_R)$ has the natural $\ZZ^d$-grading. 
On the other hand, since  $W_R \cong \omega_R$ in $\Mod R$ now, we have $\Hom_R(W_R, \omega_R) \cong R$ in $\Mod R$. 
Since the unit group of $R$ is $\kk \setminus \{ 0 \}$, the way to equip the (ungraded) module $R$ with a $\ZZ^d$-grading is unique up to a  shift. 
Hence there is $a \in \ZZ^d$ such that $\Hom_R(W_R, \omega_R) \cong R(-a)$ in $\*Mod R$. 
We use $a$ in this meaning throughout this proof. 

By \cite[Proposition~3.3.18]{BH}, $R/W_R$ is a Gorenstein ring of dimension $d-1$ and $\Ext_R^1(R/W_R, \omega_R) \cong R/W_R$ in $\Mod R$. 
By an argument similar to the above, these are isomorphic even in $\*Mod R$ 
up to a degree shift. Since $\Hom_R(W_R,\omega_R) \cong R(-a)$ 
in $\*Mod R$,  the short exact sequence $0 \too W_R \too R \too R/W_R \too 0$ yields 
\begin{equation}\label{hihihi}
\Ext_R^1(R/W_R, \omega_R) \cong (R/W_R)(-a).
\end{equation}
 
Note that $\cpx J_{R/W_R} := \Hom_R^\bullet(R/W_R, \cpx J_R)$ is the $\ZZ^d$-graded dualizing complex of $R/W_R$, and 
\begin{equation}\label{huhuhu}
H^{-d+1}(\cpx J_{R/W_R}) \cong \Ext_R^1(R/W_R, \omega_R)
\end{equation} 
in $\*Mod R$. Since $$\Hom_R(R/W_R, \EE(\kk[\M_F]))=
\begin{cases}
0 & \text{if $F=\cC(\M)$,}\\
\EE(\kk[\M_F]) & \text{if $F$ is a proper face of $\cC(\M)$,}
\end{cases}$$ 
$\cpx J_{R/W_R}$ coincides with  
the brutal truncation $J_R^{> -d}$ of $\cpx J_R$ (for this assertion, we do not use any assumption on $R=\kk[\M]$).

Let $\upl R=\kk[\upl \M]$ be the seminormalization of $R$. Since 
$$(J^i_{R/W_R})_{\cC(\M)} = (J_R^i)_{\cC(\M)}=\upl I_{\upl R}^i$$ for all $i > -d$, we have 
$$(\cpx J_{\upl R/W_{\upl R}})_{\cC(\M)}=  \upl I^{> -d}_{\upl R} = (\cpx J_{R/W_R})_{\cC(\M)},$$
where $\cpx J_{\upl R/W_{\upl R}}$  is the $\ZZ^d$-graded dualizing complex of $\upl R/W_{\upl R}$.  
Hence we have 
$$[H^{-d+1}(\cpx J_{R/W_R})]_{\cC(\M)} \cong [H^{-d+1}(\cpx J_{\upl R/W_{\upl R}})]_{\cC(\M)} 
\cong [\Ext^1_R(\upl R/W_{\upl R}, \omega_R)]_{\cC(\M)}.$$  

If $\upl R$ is {\it normal}, then $W_{\upl R}$ is its canonical module,  and 
 $$[H^{-d+1}(\cpx J_{R/W_R})]_{\cC(\M)} \cong \Ext^1_R(\upl R/W_{\upl R}, \omega_R)
\cong \upl R/W_{\upl R}.$$ 
In general, there might be gap between $[H^{-d+1}(\cpx J_{R/W_R})]_{\cC(\M)}$  and $\upl R/W_{\upl R}$,  
but an easy computation shows  that $H^{-d+1}(\cpx J_{R/W_R})$   
still contains a submodule which is isomorphic to $\upl R/ W_{\upl R}$ in $\*Mod R$.  
(Note that $[H^{-d+1}(\cpx J_{R/W_R})]_{\cC(\M)}$ is isomorphic to the kernel of  $\partial: \upl I_{\upl R}^{-d+1} 
\to \upl I_{\upl R}^{-d+2}$.) 
Combining this fact with  \eqref{hihihi} and \eqref{huhuhu}, we have a $\ZZ^d$-graded injection   
$$\upl R/W_{\upl R} \hookrightarrow (R/W_R)(-a).$$ 
This implies that  $a=0$, and hence $W_R\cong \omega_R$ in $\*Mod R$.  
Since $H_\m^d(R)_b \, (=(\omega_R)_{-b} = (W_R)_{-b}) \ne 0$ implies $b \in -\cC(M)$, $R$ is seminormal by Theorem~\ref{RRBLR}.  

Since $R$ is seminormal, we have 
$$\M \cap \int (\cC(\M)) = \ZZ \M \cap \int (\cC(\M)) = \overline{\M} \cap \int (\cC(\M)),$$
and $W_R$ coincides  with the canonical module $\omega_{\baR} \ (=W_{\baR})$ of $\baR$, 
where $\baR=\kk[\bM]$ with $\bM=\ZZ \M \cap \cC(\M)$ is the normalization of $R$. 
Hence we have  
$$\baR \cong \Hom_R(\omega_{\baR}, \omega_R) = \Hom_R(W_R, \omega_R)\cong \Hom_R(\omega_R, \omega_R) \cong R$$ 
in $\Mod R$. 
Hence $\baR \cong R$ and $R$ is normal.  
\end{proof}

\begin{rem}\label{not subcomplex}
Let $\baR=\kk[\bM]$ be the normalization of $R=\kk[\M]$. 
For a face $F$ of $\cC(\M)$, $\ZZ \M_F$ is a sublattice of $\ZZ \bM_F$, and hence 
$\kk[\ZZ \M_F \cap F]$ is a direct summand of $\kk[\bM_F]$ as an $R$-module. 
So $\upl I_R^i$ is a submodule (actually,  a direct summand) of $I_{\baR}^i$ for each $i$, but it 
does {\it not} mean $\upl \cpx I_R$ is a subcomplex of $\cpx I_{\baR}$. 

For example, consider the seminormal semigroup $\M$ given in Example~\ref{seminormal, but nonnormal}.    
Then  $R$ is of the form $\kk[x^2, y,xy]$. 
In this case,  $\upl I_R^{-2} = \kk[x,y]$, $\upl I_R^{-1}= \kk[x^2] \oplus \kk[y]$, and the degree 
$(1,0)$ component of $\partial: \upl I_R^{-2} \to \upl I_R^{-1}$ is the zero map. 
On the other hand, the normalization $\baR$ of $R$ is $\kk[x,y]$. 
Hence  $\upl I_{\baR}^{-2} = \kk[x,y]$, $\upl I_{\baR}^{-1}= \kk[x] \oplus \kk[y]$, and the degree 
$(1,0)$ component of $\partial: \upl I_{\baR}^{-2} \to \upl I_{\baR}^{-1}$ is non-zero.

Anyway, this phenomena makes the proof of Theorem~\ref{main} below complicated.  
\end{rem}

\section{Preliminaries on toric face rings}
Let $\cell$ be a finite regular CW complex with the 
intersection property, and $X$ its underlying topological space. 
More precisely, the following conditions are satisfied.   
\begin{enumerate}
\item $\emptyset \in \cell$ (for the convenience, we set $\dim \emptyset = -1$), 
$X = \bigcup_{\sigma \in \cell } \sigma$, and the cells $\sigma \in \cell $ are pairwise disjoint;
\item If $\emptyset \ne \sigma \in \cell $, then, for some $i \in \NN$, 
there exists a homeomorphism from the $i$-dimensional ball 
$\{ x \in \RR^i \mid ||x|| \leq 1 \}$ to the closure 
$\overline{\sigma}$ of $\sigma$ which maps $\{ x \in \RR^i \mid ||x|| < 1 \}$ onto $\sigma$; 
\item For $\sigma \in \cell$, the closure 
$\overline{\sigma}$ is the union of some cells in $\cell$; 
\item For $\s,\t \in \cell$, there is a cell $\u \in \cell$ such that 
$\overline \u = \overline \s \cap \overline \t$ (here $\u$ can be $\emptyset$).
\end{enumerate}  

\smallskip

We regard $\cell$ as a partially ordered set ({\it poset} for short) by $\s \geq \t \stackrel{\text{def}}
{\Longleftrightarrow} \overline \s \supset \t$. 

The following definitions of conical complexes and monoidal complexes are taken from \cite{OY}, and equivalent to the original ones in  
Bruns, Koch and R\"omer \cite{BKR} under the assumption that the cones $C_\s$ contain no line (equivalently, 
the semigroups $\M_\s$ are all positive).  
However, the notation has been changed a little from that of \cite{OY} for the usages in the present paper.

\begin{dfn}\label{sec:cell_cpx_ver}
A {\it conical complex} $(\C, \cell, \{ \i_{\s,\t}\} )$ on $\cell$ consists of the following data. 
\begin{enumerate}
\item[(0)] To each $\s \in \cell$, we assign an Euclidean space $\bE_\s= \RR^{\dim \s+1}$. 
\item[(1)] $\C = \{ \, C_\s \mid \s \in \cell \, \}$, where 
$C_{\s} \subset \bE_\s = \RR^{\dim \s + 1}$ is a polyhedral cone 
with $\dim C_{\s} = \dim \s + 1$. Here each cone $C_\s$ contains no line. 
\item[(2)] The injection $\i_{\s,\t}:C_{\t} \to C_{\s}$ for $\s, \t \in \cell$ with 
$\s \geq \t$ satisfying the following. 
\begin{enumerate}
\item $\i_{\s,\t}$ can be lifted to a linear map  
$\tilde{\i}_{\s,\t} :\bE_\t \to \bE_\s$. 
\item  The image $\i_{\s, \t}(C_{\t})$ is a face of $C_{\s}$.  
Conversely, for a face $C'$ of $C_\s$, there is a {\it sole} cell $\t$ with $\t \leq \s$
such that $\i_{\s,\t}(C_\t) = C'$. 
\item $\i_{\s, \s} = \idmap_{C_{\s}}$ and $\i_{\s,\t} \circ \i_{\t,\u}$ = $\i_{\s,\u}$ for $\s,\t,\u \in \cell$ with $\s \geq \t \geq \u$.
\end{enumerate}
\end{enumerate}
\end{dfn}

A polyhedral fan $\Sigma$ in $\RR^n$ gives a conical complex. 
In this case, as an underlying CW complex, we can take 
$\{ \, \int(C \cap \sph^{n-1}) \mid C \in \C \, \}$, where $\sph^{n-1}$ 
is the unit sphere in $\RR^n$, and the injections $\i_{\s, \t}$ are inclusion maps.

\begin{exmp}
Consider the following  cell decomposition of a M\"obius strip. 
\begin{figure}[h]\label{sec:Moebius}
$
\xy /r2.0pc/:,
{\xypolygon3"A"{~={75}~:{(-1,1.7)::}~>{}\bullet}},
+(.8,.8),
{\xypolygon3"B"{~={75}~:{(-1,1.7)::}~>{}\bullet}},
{"A1"\PATH~={**@{-}}'"A2"'"A3"'"B3"'"B2"'"B1"'"A1"},
"A2";"B2"**@{-},
{\vtwist~{"A1"}{"B1"}{"A3"}{"B3"}},
"A1"*+!RD{x}, "A2"*+!R{y}, "A3"*+!LU{z},
"B1"*+!RD{u}, "B2"*+!R{v}, "B3"*+!LU{w}
\endxy
$
%\caption{M\"obius strip}
\end{figure}
\noindent Regarding each rectangles as the cross-sections of 3-dimensional cones,
we have a conical complex that is not a fan (see \cite[Example~1.36]{BG}).
\end{exmp}

Let $\bL_\s$ be the set of lattice points  $\ZZ^{\dim \s +1}$ of $\bE_\s = \RR^{\dim \s +1}$. 
Assume that 
$ \tilde{\i}_{\s, \t}(\bL_\t) = \tilde{\i}_{\s, \t}(\bE_\t) \cap \bL_\s$ for all $\s, \t \in \cell$ with $\s \geq \t$.

\begin{dfn}\label{monoidal complex}
A \term{monoidal complex} supported by a conical complex $(\C,\cell, \{ \i_{\s,\t}\})$ is a set of monoids
$\MM=\mbra{\M_\s}_{\s \in \cell}$ with the following conditions:
\begin{enumerate}
\item $\M_\s \subset \bL_\s = \ZZ^{\dim \s+1}$ for each $\s\in \cell$, 
and it is a finitely generated additive submonoid (so $\M_\s$ is an affine semigroup);
\item $\M_\s \subset C_\s$ and $\RR_{\geq 0} \M_\s = C_\s$ for each $\s\in \cell$; 
\item for $\s,\t \in \cell$ with $\s \ge \t$, the map $\i_{\s,\t}:C_\t \to C_\s$ induces an isomorphism
$\M_\t \cong \M_\s \cap \i_{\s,\t}(C_\t)$ of monoids.
\end{enumerate}
\end{dfn}

If $\C$ is a rational fan in $\RR^n$, then $\set{C \cap \ZZ^{n}} {C \in \C}$ gives 
a monoidal complex. More generally, taking submonoids of $C \cap \ZZ^{n}$ carefully,  
we can get a ``non-normal" monoidal complex. 

\medskip

For a monoidal complex $\MM=\mbra{\M_\s}_{\s \in \cell}$, set
$$
\sM := \colimit_{\s \in \cell}\M_\s, 
$$
where the direct limit is taken with respect to  $\i_{\s, \t} : \M_\t \to \M_\s$
 for $\s, \t \in \cell$ with $\s \ge \t$.
Note that $\sM$ is just a set and no longer a monoid in general.  
Since all $\i_{\s,\t}$ are injective, we can regard $\M_\s$ as a subset of $\sM$. 
For example, if $\{ \M_\s \}_{\s \in \cell}$ comes from a fan in $\RR^n$, then $\sM = \bigcup_{\s \in \cell} \M_\s \subset \ZZ^n$.

Let $a, b \in \sM$. If there is some $\s \in \cell$ with $a, b \in C_\s$, there is a unique minimal cell among 
these $\s$'s. (In fact, if $C_{\s_1}, C_{\s_2} \in \cell$ contain both $a$ and $b$, there is a cell $\t \in \cell$ with 
$\overline \t = \overline \s_1 \cap \overline \s_2$ by our assumption on $\cell$, and $C_\t$ contains both $a$ and $b$.) 
If $\s$ is the minimal one with this property, we have $a, b \in \M_\s$ and we can define $a + b \in \M_\s \subset \sM$. 
If there is no $\s \in \cell$ with $a,b \in C_\s$, then $a + b$ does not exist.

\begin{dfn}[\cite{BKR}]\label{non-multi-graded }
Let $\{ \M_\s \}_{\s \in \cell}$ be a monoidal complex with $\sM := \colimit \M_\s$, and $\kk$ a field. Then the $\kk$-vector space
$$
\fring \MM := \bigoplus_{a \in \sM} \kk \, x^a,
$$
where $x$ is a variable, equipped with the following multiplication
$$
x^a \cdot x^b = \begin{cases}
x^{a+b} & \text{if $a+b$ exists,}\\
0       & \text{otherwise,}
\end{cases}
$$
has a $\kk$-algebra structure. 
We call $\fring \MM$ the \term{toric face ring} of $\MM$ over $\kk$.
\end{dfn}

Clearly,  $\dim \kk[\MM] = \dim \cell +1$. In the rest of this paper, we set $d:= \dim \kk[\MM]$.  
Stanley-Reisner rings and affine semigroup rings (of 
positive semigroups) can be established as toric face rings.
If $\MM$ comes from a fan in $\RR^n$, then $\kk[\MM]$ admits a $\ZZ^n$-grading 
with $\dim_\kk \kk[\MM]_a \leq 1$ for all $a \in \ZZ^n$. 
But this is not true in general.

\begin{exmp}[{\cite[Example 4.6]{BKR}}]\label{sec:Moebius_ring}
Consider the conical complex in Example \ref{sec:Moebius}. 
Assigning  normal semigroup rings of the form $\kk[a,b,c,d]/(ac-bd)$ 
to each rectangles, we have a toric face ring of the form 
$$\kk[x, y, z, u, v, w]/(xv - uy, vz - yw, 
xz - uw, uvw, uvz),$$
which does not admit a nice multi-grading. 
We can also get a similar example whose $\kk[\M_\s]$ are not normal. 
\end{exmp}

We say a toric face ring $R=\kk[\MM]$ is {\it cone-wise normal}, if $\kk[\M_\s]$ is normal for 
all $\s \in \cell$.  
The notion of cone-wise normal toric face rings coincides with that of 
the ring $\kk[{\mathcal WF}]$ associated with  a {\it weak fan} ${\mathcal WF}$
introduced by  Bruns and Gubeladze \cite{BG02}. 
They gave an example of a cone-wise normal toric face ring which does not admit 
a $\ZZ$-grading with $R_0=\kk$ (\cite[Example~2.7]{BG02}).

\medskip

For $\s \in \cell$,  a monomial ideal $\p_\s := (x^a \mid a \in \sM \setminus \M_\s)$ of $R$ 
is prime. In fact, the quotient ring $R/\p_\s$ is isomorphic to the affine semigroup ring 
$\kk[\M_\s]$.  
We regard $\kk[\M_\s]$ as an $R$-module, through $R/\p_\s \cong \kk[\M_\s]$.

 Set 
$$I^{-i}_R := \bigoplus_{\substack{\s \in \cell \\ \dim \s = i-1}} \kk[\M_\s]$$
for $i = 0, \dots , d$, and define $\partial:I^{-i}_R \to I^{-i+1}_R$
by
$$
\partial(y) = \sum_{\substack{\dim \t =i-2 \\ \t \leq \s}}\e(\s,\t) \cdot 
\pi_{\t,\s}(y)
$$
for $y \in \kk[\M_\s] \subset I^{-i}_R$,
where  $\pi_{\t,\s}$ is the natural surjection $\kk[\M_\s] \to \kk[\M_\t]$
(note that if $\t \leq \s$ then $\p_\s \subset \p_\t$) and 
$\e$ is an incidence function of $\cell$. Then
$$
\cpx I_R: 0 \longto I^{-d}_R \longto I^{-d + 1}_R 
\longto \cdots \longto I^0_R \longto 0
$$
is a cochain complex of  finitely generated $R$-modules. 
The following is the main result of \cite{OY}.  

\begin{thm}[{\cite[Theorem~5.2]{OY}}]\label{sec:ishida}
If $R$ is cone-wise normal, then 
$\cpx I_R$ is quasi-isomorphic to the dualizing complex $\cpx D_R$ of $R$.
\end{thm}

The proof of the main result in the next section largely depends on (the proof of) Theorem~\ref{sec:ishida}, but the proof in \cite{OY} is long and technical.  
So we summarize it here for the reader's convenience. See \cite{OY} for details. 

\bigskip

\noindent{\it An outline of the proof of Theorem~\ref{sec:ishida}.}
To prove the theorem, we realize $\cpx I_R$ as a subcomplex of $\cpx D_R$.  
Set $c(\s):=\dim \s +1 =\dim \kk[\M_\s]$ for a cell $\s$. 
The proof is divided into three steps. 
%Since $R$ is cone-wise normal, we may assume that $\ZZ \M_\s = \bL_\s$ for all $\s \in \cell$. 

\medskip

\noindent{\it Step 1.} We have a canonical injection $i_\s:\kk[\M_\s] \hookrightarrow D_R^{-c(\s)}$.

\medskip

We fix a cell $\s$, and set $c:=c(\s)$. 
Since $\kk[\M_\s]$ is normal, it is Cohen-Macaulay and admits the canonical module simply denoted by $\omega_\s$.   
Note that 
$$ H^{-c}(\Hom_R(\omega_\s, \cpx D_R)) = \Ext_R^{-c}(\omega_\s, \cpx D_R) \cong \kk[\M_\s].$$
Since $\Hom_R(\omega_\s, D_R^{-c-1})=0$, the cohomology $H^{-c}(\Hom_R(\omega_\s, \cpx D_R))$ is the kernel of the map 
\begin{equation}\label{partial_*} 
\Hom_R(\omega_\s, \partial_{\cpx D_R}) : \Hom_R(\omega_\s, D_R^{-c}) \too \Hom_R(\omega_\s, D_R^{-c+1}).  
\end{equation}
Through the identification, 
$$\Hom_R(\omega_\s, D_R^{-c})=
\Hom_R(\kk[\M_\s], D_R^{-c})\\
\cong \{ \, y \in D_R^{-c} \mid \p_\s y=0 \, \},$$
the kernel of the map \eqref{partial_*} is 
$$i_\s (\kk[\M_\s]) := \{\, y \in D_R^{-c} \mid \p_\s y=0 \ \text{and} \ \partial_{\cpx D_R}(\q_\s y)=0 \, \},$$
where $\q_\s$ is the set  $\{ \, x^a \in R \mid  a \in (\M_\s \cap \int(C_\s))\, \}$. 
(Note that $\omega_\s$ is the ideal of $\kk[\M_\s]$ generated by $\q_\s$.) 
Clearly, $i_\s(\kk[\M_\s]) \cong \kk[\M_\s]$. 

Of course, we just chose the {\it subset} $i_\s(\kk[\M_\s])$ of $D_R^{-c}$, not an injection 
$i_\s:\kk[\M_\s] \hookrightarrow D_R^{-c}$. 
However, the $R$-module $\kk[\M_\s]$ is generated by a single element, and the choice of a generator (i.e., 
 the choice of $i_\s$) is unique up to constant multiplication. 
This small ambiguity does not affect the argument below. \qed

\medskip

\noindent{\it Step 2. } $\bigoplus_{\s \in \cell} i_\s(\kk[\M_\s])$ is a subcomplex of $\cpx D_R$.

\medskip

The dualizing complex $\cpx D_\s := \cpx D_{\kk[\M_\s]}$ of $\kk[\M_\s]$ coincides with $\Hom_R(\kk[\M_\s], \cpx D_R)$, 
which can be seen as a subcomplex of $\cpx D_R$.  
Since $\kk[\M_\s]$ is $\ZZ^{c(\s)}$-graded, we have the $\ZZ^{c(\s)}$-graded dualizing complex $\cpx J_\s := \cpx J_{\kk[\M_\s]}$, and  
a quasi-isomorphism $\cpx J_\s \to \cpx D_\s$. 
Composing this morphism with $\cpx D_\s \to \cpx D_R$, we get a chain map $h_\s: \cpx J_\s \to \cpx D_R$  
which induces 
\begin{equation}\label{rururu}
H^i(\Hom_R(\omega_\s, \cpx J_\s)) \cong H^i(\Hom_R(\omega_\s, \cpx D_R)).
\end{equation}

Applying the same argument as Step 1, we have an injection 
${}^* i_{\s, \t} : \kk[\M_\t] \hookrightarrow J_\s^{-c(\t)}$ for a cell $\t$ with $\t \leq \s$.  
By \eqref{rururu}, it is easy to see that $$i_\t(\kk[\M_\t]) = h_\s \circ {}^* i_{\s, \t}(\kk[\M_\t]).$$

On the other hand, we have that 
\begin{equation}\label{relayer}
(\cpx J_\s)_{C_\s}=\bigoplus_{\t \leq \s} {}^*i_{\s,\t}(\kk[\M_\t]),
\end{equation}
where $C_\s$ is the polyhedral cone spanned by $\M_\s$.  
Since $\cpx J_\s$ is a $\ZZ^{c(\s)}$-graded complex, the right side of \eqref{relayer} is a subcomplex of $\cpx J_\s$. 
Since $h_\s$ is a chain map, $\bigoplus_{\t \leq \s} i_\s(\kk[\M_\t])$ 
forms a subcomplex of $\cpx D_R$. It implies that $\bigoplus_{\s \in \cell} i_\s(\kk[\M_\s])$ is also a subcomplex of $\cpx D_R$. \qed 

\medskip 

Since $\bigoplus_{\s \in \cell} i_\s(\kk[\M_\s])$ is isomorphic to $\cpx I_R$, 
it suffices to show the following.

\medskip

\noindent{\it Step 3.} $\cpx D_R$ is quasi-isomorphic to its subcomplex $\bigoplus_{\s \in \cell} i_\s(\kk[\M_\s])$.

\medskip

The argument for this step will be used around the proof of Theorem~\ref{Phi} 
after a slight generalization. 
There, we  explain this idea in detail, so we do not give a summary here.  
\qed

\section{Dualizing complexes of seminormal toric face rings}
We start from the following fact pointed out by Nguyen \cite{Ngu}.  

\begin{prop}[{\cite[Proposition~3.5]{Ngu}}]\label{Ngu}
For a toric face ring $\kk[\MM]$, the following are equivalent. 
\begin{itemize}
\item[(i)] $\kk[\MM]$ is seminormal. 
\item[(ii)] $\kk[\M_\s]$ is seminormal for all $\s \in \cell$. 
\end{itemize} 
\end{prop}

Recall the precise definition of a monoidal complex $\MM$ given in the previous section.  
For each $\s \in \cell$, let $\upl \M_\s \subset \bL_\s$ be the monoid constructed from $\M_\s$ by the operation in \eqref{monoid seminormalization}, 
that is, $\kk[\upl \M_\s]$ is the seminormalization of $\kk[\M_\s]$. 
Then $\upl \MM:= \{\upl \M_\s\}_{\s \in \cell}$ forms a monoidal complex, and 
$\upl R:= \kk[\upl \MM]$ is the seminormalization of $R:= \kk[\MM]$. 
In particular, $R$ is seminormal if and only if $\MM =\upl \MM$. 

On the other hand, $\kk[\ZZ \M_\s \cap C_\s]$ is the normalization 
 of $\kk[\M_\s]$ (since we do not assume 
that $\ZZ\MM_\s =\bL_\s$, we have  $\ZZ \M_\s \cap C_\s \ne \bL_\s \cap C_\s$ in general), 
but $\{\ZZ \M_\s \cap C_\s\}_{\s \in \cell}$ does {\it not} form 
a monoidal complex. The monoidal complex $\MM$ of  Example~\ref{bMM} below gives a counter example. 
In fact, the condition (3) of Definition~\ref{monoidal complex} is violated. 

\medskip

We consider the following cochain complex 
$$
\upl \cpx I_R: 0 \longto \upl I^{-d}_R \longto \upl I^{-d + 1}_R 
\longto \cdots \longto \upl I^0_R \longto 0
$$
with 
$$\upl I^{-i}_R := \bigoplus_{\substack{\s \in \cell \\ \dim \s = i-1}} \kk[\ZZ \M_\s \cap C_\s].$$
The differential map $\partial$ is given by  
$$
\partial(y) = \sum_{\substack{\dim \t =i-2 \\ \t \leq \s}}\e(\s,\t) \cdot 
\pi_{\t,\s}(y)
$$
for $y \in \kk[\ZZ \M_\s \cap C_\s] \subset I^{-i}_R$, where  $\pi_{\t,\s}$ is the natural surjection 
$\kk[\ZZ \M_\s \cap C_\s] \to \kk[\ZZ \M_\t \cap C_\t]$.
Clearly, $\upl \cpx I_R$ is a cochain complex of finitely generated $R$-modules
.

\begin{thm}\label{main}
If a toric face ring $R=\kk[\MM]$ is seminormal, then $\upl \cpx I_R$ is quasi-isomorphic to the dualizing complex $\cpx D_R$.  
\end{thm}

To prove the theorem, we need some preparation.  
For each $\s \in \cell$, set $\tM_\s:= \bL_\s \cap C_\s$. 
Then $\{ \, \tM_\s \, \}_{\s \in \cell}$ is a monoidal complex again. 
We can regard that $|\bMM| := \colimit \tM_\s$ contains $\sM$ as a subset. 

\begin{exmp}\label{bMM}
While $\kk[\tM_\s]$ is always a normal semigroup ring, it is not the normalization of $\kk[\M_\s]$. 
For example, consider the monoidal complex $\MM$  illustrated below.  
Let $\M_\s$ be the monoid corresponding to the first quadrant, then $\kk[\M_\s]=\kk[x^2,y]$ is normal, 
but we have $\kk[\tM_\s] = \kk[x,y] \supsetneq \kk[\M_\s]$. 
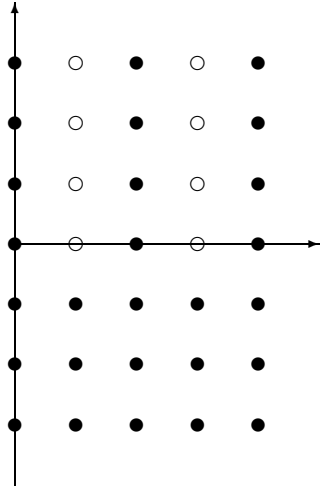
\begin{center}
\begingroup  % \setlength
\setlength\unitlength{8mm} 
\begin{picture}(7,9.5)(0,-0.5)  
\blackcirclefill<.2>(4,3)(0,1) 
\put(0,7){\circle*{.2}}  
\put(1,7){\circle{.2}}  
\put(2,7){\circle*{.2}} 
\put(3,7){\circle{.2}}
\put(4,7){\circle*{.2}} 
\put(0,6){\circle*{.2}}  
\put(1,6){\circle{.2}}  
\put(2,6){\circle*{.2}} 
\put(3,6){\circle{.2}}
\put(4,6){\circle*{.2}} 
\put(0,5){\circle*{.2}}  
\put(1,5){\circle{.2}}  
\put(2,5){\circle*{.2}} 
\put(3,5){\circle{.2}}
\put(4,5){\circle*{.2}} 
\put(0,4){\circle*{.2}}  
\put(1,4){\circle{.2}}  
\put(2,4){\circle*{.2}} 
\put(3,4){\circle{.2}}
\put(4,4){\circle*{.2}} 
%\put(5,0){\circle{.2}}
%\put(6,0){\circle*{.2}} 
%\blackcirclefill<.2>(6,0)(4,0)  % \blackcirclefill 
\put(0,4){\vector(1,0){5}}  % 
\put(0,0){\vector(0,1){8}}  % 
\end{picture}
\endgroup
\end{center}
\end{exmp}

Set $\oR:=\kk[\bMM]$. The next result holds, even if $\kk[\MM]$ is not seminormal.  

\begin{lem}\label{oR is fin. gen.}
For any $\MM$, $\oR=\kk[\bMM]$ is a finitely generated module over $R=\kk[\MM]$. 
\end{lem}

\begin{proof}
It suffices to show that  $\kk[\tM_\s]$ is finitely generated as a $\kk[\M_\s]$-module for each $\s \in \cell$. 
This must be a well-known result, but we give a proof here for the reader's convenience. 
If $\dim \s=0$, then the assertion is clear (in fact, $\kk[\tM_\s]$ is a polynomial ring 
with one variable in this case). 
If $\dim \kk[\M_\s]>1$, set $A:=\kk[\M_\s]$, and let $A'$ be the $A$-subalgebra of $\kk[\tM_\s]$ 
generated by $\{ \, x^a \mid a \in \tM_\t, \t < \s, \dim \t =0 \, \}$. 
By the above remark, $A'$ is a finitely generated  $A$-module. 
Since $\kk[\tM_\s]$ is the normalization of $A'$, it is a finitely generated as an $A'$-module, 
hence also as an $A$-module.   
\end{proof}

We regard $\kk[\tM_\s]$ as an $R$-module by the compositions of the ring homomorphisms 
$R \twoheadrightarrow R/\p_\s (\cong \kk[\M_\s]) \hookrightarrow  \kk[\tM_\s]$, which is 
the same thing as $R \hookrightarrow \oR \twoheadrightarrow \kk[\tM_\s]$.

As in the previous section, we set $c(\s):= \dim \s +1 =\dim \kk[\M_\s]$.  
For the simplicity, the dualizing complexes $\cpx D_{\kk[\M_\s]}$ (resp.  $\cpx D_{\kk[\tM_\s]}$) of 
$\kk[\M_\s]$ (resp. $\kk[\tM_\s]$) is denoted by $\cpx D_\s$ (resp.  $\cpx D_{\ols}$). 
Since both $\kk[\M_\s]$ and $\kk[\tM_\s]$ are $\ZZ^{c(\s)}$-graded, they admit the $\ZZ^{c(\s)}$-graded 
dualizing complexes $\cpx J_\s:=\cpx J_{\kk[\M_\s]}$ and $\cpx J_{\ols}:= \cpx J_{\kk[\tM_\s]}$ respectively. 
Similarly, we also set $\upl \cpx I_\s := \upl \cpx I_{\kk[\M_\s]}$ and  $\cpx I_{\ols} := \cpx I_{\kk[\tM_\s]} (= \upl \cpx I_{\kk[\tM_\s]})$ 
for the simplicity.

Since $\oR$ is cone-wise normal, $\cpx I_{\oR}$ is quasi-isomorphic to $\cpx D_{\oR}$ 
by Theorem~\ref{sec:ishida}. Moreover, we have the following.

\begin{lem}\label{psi_s}
There is a quasi-isomorphism $\psi: \cpx I_{\oR} \to \cpx D_{\oR}$ such that the 
induced map $\psi_\s := \inHom_{\oR}(\kk[\tM_\s], \psi) : \cpx I_{\ols} \to 
\cpx D_{\ols}$ is a quasi-isomorphism for all $\s \in \cell$.  
\end{lem}

\begin{proof}
This fact has been shown in the proof of \cite[Theorem~5.2]{OY}(Theorem~\ref{sec:ishida} 
of the present paper). Recall the outline of the proof introduced in the previous section. 
\end{proof}

Since $\oR$ is finitely generated as an $R$-module by Lemma~\ref{oR is fin. gen.}, 
we have $\DC_{\oR} = \inHom_R(\oR, \DC_R)$. 
Via the canonical injection $R \hookrightarrow \oR$, we have a chain map 
$$\lambda :\DC_{\oR} = \inHom_R(\oR, \DC_R) \longrightarrow \inHom_R(R,\DC_R)=\DC_R.$$ 
Similarly, for each $\s$, the injection $\kk[\M_\s] \hookrightarrow \kk[\tM_\s]$ 
induces a chain map $\lambda_\s : \cpx D_{\ols} \to \cpx D_{\s}$. 
Since $\kk[\tM_\s]$ is a finitely generated  $\ZZ^{c(\s)}$-graded module over 
$\kk[\M_\s]$ and $\cpx J_\s$ is the dualizing complex in the $\ZZ^{c(\s)}$-graded context, we have 
  $\Hom_{\kk[\M_\s]}^\bullet(\kk[\tM_\s], \cpx J_\s)=\cpx J_{\ols}.$ 
The injection $\kk[\M_\s] \hookrightarrow \kk[\tM_\s]$ induces the $\ZZ^{c(\s)}$-graded chain map
$\mu_\s' : \cpx J_{\ols} \too \cpx J_\s.$

Note that  $\M_\s$ and $\tM_\s$ span the same polyhedral cone $C_\s$.  
Since $\kk[\M_\s]$ is seminormal and $\kk[\tM_\s]$ is normal, we have $\cpx J_\s \cong (\cpx J_\s)_{C_\s} 
=\upl \cpx I_\s$ and  $\cpx J_{\ols} \cong (\cpx J_{\ols})_{C_\s} = \upl \cpx I_{\ols}=\cpx I_{\ols}$ 
as shown in the proof of Theorem~\ref{+I_R}. Taking the $C_\s$-graded part of $\mu'_\s$, we have the chain map 
$$\mu_\s : \cpx I_{\ols} \too \upl \cpx I_\s.$$

\begin{lem}\label{commutative diagram}
For the quasi-isomorphism $\psi_\s : \cpx  I_{\ols} \to \cpx D_{\ols}$ of Lemma~\ref{psi_s}, 
we have a quasi-isomorphism $\phi_\s :  \upl\cpx I_\s \to \cpx D_\s$ 
which makes the following diagram commutative.
$$
\xymatrix{
I^\bullet_{\ols} \ar[r]^{\psi_\sigma} \ar[d]_{\mu_\sigma} %\ar@{}[rd]_{\circlearrowright}
& D^\bullet_{\ols} \ar[d]^{\lambda_\sigma} \\
\upl I^\bullet_{\s} \ar[r]_{\phi_\sigma} & D^\bullet_\s
}
$$ 
\end{lem}

\begin{proof}
It is easy to see that there exists a quasi-isomorphism $\psi_\s':  \cpx  J_{\ols} \to \cpx D_{\ols}$  
which is an extension of  $ \psi_\s: \cpx I_{\ols} \to  D^\bullet_{\ols}$.  
Since $\mu_\s ': \cpx I_{\ols}  \too \upl \cpx I_\s$ is the restriction of $\mu_\s': \cpx J_{\ols}  \too \cpx J_\s$, it suffices to construct 
a quasi-isomorphism $\phi'_\s :  \cpx J_\s \to \cpx D_\s$ with 
$$
\xymatrix{
\cpx J_{\ols} \ar[r]^{\psi'_\sigma} \ar[d]_{\mu'_\sigma} %\ar@{}[rd]_{\circlearrowright}
& D^\bullet_{\ols} \ar[d]^{\lambda_\sigma} \\
\cpx  J_\s \ar[r]_{\phi'_\sigma} & D^\bullet_\s. 
}
$$ 
In fact, the restriction of $\phi'_\s$ to $\upl I^\bullet_{\s}$ gives  $\phi_\s$ satisfying the expected condition. 

Since $\cpx J_\s \cong D^\bullet_\s$ in $\Db(\Mod \kk[\M_\s])$, we have 
a quasi-isomorphism $\xi : \cpx J_\s \to \cpx D_\s$. 
Taking $\Hom_{\kk[\M_\s]}(\kk[\tM_\s], -)$, we get  a chain map  
$$\xi_*: \cpx J_{\ols} = \Hom_{\kk[\M_\s]}(\kk[\tM_\s], \cpx J_\s) \longrightarrow 
\Hom_{\kk[\M_\s]}(\kk[\tM_\s], \cpx D_\s) = \cpx D_{\ols}.$$
Note that $\cpx J_\s$ is a cochain complex of injective objects in the category $\*Mod (\kk[\M_\s])$ 
of $\ZZ^{c(\s)}$-graded $\kk[\M_\s]$ modules, and $\kk[\tM_\s] \in \*Mod (\kk[\M_\s])$.  
Hence $\xi_*$ is a quasi-isomorphism. 

Clearly, $\xi_*$ is $\kk[\tM_\s]$-linear, and can be extended to a $\kk[\tM_\s]$-linear automorphism 
$\overline{\xi}_*$ of $\cpx D_{\ols}$ uniquely (of course, the same is true for $\psi_\s'$).  
Since  $$\Hom_{\Db(\Mod \kk[\tM_\s])}(\cpx D_{\ols}, 
\cpx D_{\ols}) = \kk[\tM_\s]$$ and $\cpx D_{\ols}$ is a cochain complex of injective modules, the automorphism 
$\overline{\xi}_*$ is homotopic to the multiplication by $c$ for some $0 \ne c \in \kk$. 
Moreover, since $\cpx D_{\ols}$ is of the form \eqref{normal form},  $\overline{\xi}_*$ is {\it equal}
to the multiplication by $c$. 
Since the same is true for  $\psi'_\s$, we have $\psi'_\s = c' \, \xi_*$ for some $0 \ne c' \in \kk$. 
Hence $\phi_\s' := c' \, \xi$ satisfies the desired condition. 
\end{proof}

For each $i \in \ZZ$,  $\upl I^i_R$ is an $R$-submodule of $I^i_{\oR}$. 
However $\upl \cpx I_R$ is {\it not} a subcomplex of $\cpx I_{\oR}$.  
This problem occurs even in the semigroup ring case. See Remark~\ref{not subcomplex}.

Let $\kappa:  \upl \cpx I_R \dashrightarrow \cpx I_{\oR}$ be 
the collection of the natural injections  $\upl I^i_R \hookrightarrow I^i_{\oR}$
(since this is not a chain map, we use the 
symbol ``$\dashrightarrow$"). The similar map $\kappa_\s: \upl \cpx I_\s 
\dashrightarrow \cpx I_{\ols}$ is not a chain map in general again.  
For each $i$, $\upl I^i_\s$ is a direct summand of $I_{\ols}^i$ as a $\kk[\M_\s]$-module,    
the $i$-th component $\mu^i_\s: I_{\ols}^i \too \upl I_\s^i$ of the chain map 
$\mu_\s:  \cpx I_{\ols} \too \upl \cpx I_\s$ satisfies  
$\mu_\s^i \circ \kappa^i_\s = \operatorname{Id}$. 

\begin{lem}\label{composition}
The composition 
$\upl \cpx I_R \stackrel{\kappa}{\dashrightarrow} \cpx I_{\oR} \stackrel{\psi}{\longrightarrow} \cpx D_{\oR} 
\stackrel{\lambda}{\longrightarrow} \cpx D_R$ 
is a chain map. 
\end{lem}

\begin{proof}
It suffice to check that 
$$\partial^{i+1}_{\cpx D_R}\circ (\lambda^i \circ \psi^i \circ \kappa^i)(y)=
(\lambda^{i+1} \circ \psi^{i+1} \circ \kappa^{i+1}) \circ \partial^i_{ \upl \cpx I_R}(y)$$
for all ``homogeneous" element $y$ (i.e., $y \in (\upl I_R^i)_a$ for some $a \in \sM$), 
since any element of $\upl I_R^i$ is a sum of these elements.  
Then we can regard $y \in \upl I^i_{\s}$ for some $\s \in \cell$. 
We have the following commutative diagram.
$$
\xymatrix{
\upl I^i_\s \ar[r]^{\kappa^i_\sigma} \ar@{^(->}[d] & I^i_{\ols} \ar[r]^{\psi^i_\sigma} \ar@{^(->}[d] &
D^i_{\ols} \ar[r]^{\lambda^i_\sigma} \ar@{^(->}[d] & D^i_\s  \ar@{^(->}[d]\\
\upl I^i_R \ar[r]_{\kappa^i} & I^i_{\oR} \ar[r]_{\psi^i} & D^i_{\oR} \ar[r]_{\lambda^i} & D^i_R
}
$$
The commutativity of the left square is clear, that of the middle one is Lemma~\ref{psi_s}, and 
that of the right one follows from the fact that the composition $R \hookrightarrow \oR \twoheadrightarrow \kk[\tM_\s]$ 
coincides with the composition $R \twoheadrightarrow \kk[\M_\s] \hookrightarrow \kk[\tM_\s]$. 

By Lemma~\ref{commutative diagram}, we have $\lambda^i_\s \circ \psi^i_\s \circ \kappa^i_\s 
= \phi_\s^i \circ \mu^i_\s \circ \kappa^i_\s = \phi_\s^i$. 
Since $\phi_\s$ is a chain map, we are done. 
\end{proof}

Let $\phi$ denote the chain map $\cpx J_R \to \cpx D_R$ constructed in Lemma~\ref{composition}. 
To prove Theorem~\ref{main}, we will show that $\phi$
is a quasi-isomorphism by a slightly indirect way.

\begin{dfn}
Let $R = \fring \MM$ be a toric face ring. 
We say an $R$-module $M$ is {\it $|\bMM|$-graded} if the following are satisfied;
  \begin{itemize}
  \item[(i)] $M = \bigoplus_{a \in |\bMM|}M_a$ as $\kk$-vector spaces;
  \item[(ii)] Let $a \in \sM$ and $b \in |\bMM|$. If $a+b$ exists (equivalently, $a, b \in \tM_\s$ for some $\s \in \cell$),  
then $x^a \, M_b \subset M_{a + b}$.  Otherwise, $x^a \, M_b = 0$.
  \end{itemize}
Let $\Mbgr R$ denote the subcategory of $\Mod R$ whose objects are $|\bMM|$-graded
and homomorphisms are $f:M \to N$ with $f(M_a) \subset N_a$ for all $a \in |\bMM|$. 

We say $M \in \Mbgr R$ is {\it $\sM$-graded}, if $M = \bigoplus_{a \in \sM}M_a$. 
Let $\Mgr R$ denote the subcategory of $\Mbgr R$ consisting of  $\sM$-graded modules. 
\end{dfn}

Clearly, $\Mbgr R$ and $\Mgr R$ are abelian categories. 
It is easy to see that $R \in \Mgr R$ and $\oR \in \Mbgr R$. 
Moreover, $\cpx I_R$ (resp. $\upl \cpx I_R$) is a cochain complex in $\Mgr R$ (resp. $\Mbgr R$).

\begin{dfn}\label{Sq}
For each $a \in |\bMM|$, there is a unique cell  
$\s \in \cell$ with $a \in \int(C_\s)$ (equivalently, $a \in \tM_\s$ and $\s$ is the minimal 
one with this property). 
This cell $\s$ is denoted by $\supp(a)$. 

An $R$-module $M \in \Mod R$ is said to be \term{squarefree} if it is $\sM$-graded 
({\it not} $|\bMM|$-graded), finitely generated,  and the multiplication map $M_a \ni  y \longmapsto x^b y \in M_{a+b}$  
is bijective for all $a,b\in \sM$ with $\supp(a) \supset \supp(b)$.
\end{dfn}

For example, $\kk[\M_\s]$ and $R$ itself are squarefree $R$-modules.  
In \cite{OY}, squarefree modules over a cone-wise normal toric face ring play a key role. 
Many properties are lost in the non-normal case. 
For example, $\upl \cpx I_R$ is no longer a complex of squarefree modules. 
In fact, $\upl I_R^i$ is $|\bMM|$-graded, not $\sM$-graded. 
However,  the next result still holds. 

\begin{lem}[{c.f. \cite[Lemma~4.2]{OY}}]\label{Sq lem}
Let $\Sq R$ be the full subcategory of $\Mgr R$ consisting of squarefree modules. 
Then $\Sq R$ is an abelian category with enough injectives, 
and indecomposable injectives are objects isomorphic to $\kk[\M_\s]$ for some $\s \in \cell$.  
The injective dimension of any object is at most $d$.
\end{lem}

The proof is similar to the cone-wise normal case (\cite{OY}), and we omit it here. 
We just remark that $\Sq R$ is equivalent to the category of 
finitely generated left $\Lambda$-modules, where $\Lambda$ 
is the incidence algebra of $\cell$ (as a poset) over $\kk$. 
 
\medskip

Let $\InjSq$ be the full subcategory of $\Sq R$ consisting of all injective objects, that is, 
finite direct sums of copies of $\kk[\M_\s]$ for various $\s \in \cell$. 
Then the bounded homotopy category $\Cb(\InjSq)$ 
is equivalent to $\Db(\Sq R)$. We have an exact functor 
$$\inHom_R(-, \upl \cpx I_R): \Cb(\InjSq) \to \Db(\Mod R)^\op.$$ 

Similarly, we have an  exact functor 
$$\inHom_R(-, \cpx D_R): \Cb(\InjSq) \to \Db(\Mod R)^\op.$$
The chain map $\phi: \upl \cpx I_R \to \cpx D_R$ gives a natural transformation 
$$\Phi: \inHom_R(-, \upl \cpx I_R) \to \inHom_R(-, \cpx D_R).$$

\begin{thm}\label{Phi}
If $R$ is seminormal, $\Phi$ is a natural isomorphism. 
\end{thm}

\begin{proof}
By virtue of \cite[Proposition 7.1]{Ha}, it suffices to show that 
$$\Phi(\kk[\M_\s]): \upl \cpx I_\s = \inHom_R(\kk[\M_\s], \upl \cpx I_R) \to \inHom_R(\kk[\M_\s], \cpx D_R) = \cpx D_\s$$  
is a quasi-isomorphism for all $\s \in \cell$.    Since $\Phi(\kk[\M_\s]) = \inHom_R(\kk[\M_\s], \phi)$, it is factored as 
$\upl \cpx I_\s \stackrel{\kappa_\s}{\dashrightarrow} \cpx I_{\ols} \stackrel{\psi_\sigma}{\longrightarrow} \cpx D_{\ols} 
\stackrel{\lambda_\s}{\longrightarrow} \cpx D_\s$.  As shown in the proof of Lemma~\ref{composition}, 
this coincides with the quasi-isomorphism $\phi_\s$ of Lemma~\ref{commutative diagram}.  
\end{proof}

\noindent{\it The proof of Theorem~\ref{main}.}
The assertion follows from Theorem~\ref{Phi}. In fact, since $R \in \Sq R$, we have an isomorphism 
$\Phi(R): \inHom_R(\cpx E, \upl \cpx I_R) \to \inHom_R(\cpx E, \cpx D_R)$, where $\cpx E$ is 
an injective resolution of $R$ in $\Sq R$. It is clear that $ \inHom_R(\cpx E, \cpx D_R) \cong \inHom_R(R, \cpx D_R)  \cong \cpx D_R$, but we can 
also show that  $\inHom_R(\cpx E, \upl \cpx I_R)  \cong  \upl \cpx I_R$ by the usual double complex argument. The key fact is that
 $\inHom_R(\cpx E, \upl I^i_R)$  is an acyclic complex whose 0th cohomology is $\upl I^i_R$ for each $i$. To see this, note that  
an indecomposable components of $\cpx E$ and $\upl I^i_R$ are $\kk[\M_\s]$ and $\kk[\ZZ \M_\tau \cap C_\tau]$ 
respectively for some $\s, \tau \in \cell$, moreover 
$$\Hom_R(\kk[\M_\s], \kk[\ZZ \M_\tau \cap C_\tau]) \cong \begin{cases}
\kk[\ZZ \M_\tau \cap C_\tau] & \text{if $\s \ge \tau$,}\\
0 & \text{otherwise.}
\end{cases}$$ 
Take $a \in \sM$ with $\supp(a) =\t$. Then, $\s \ge \t$ if and only if $\kk[\M_\s]_a \ne0$. Hence we have 
$\inHom_R(\cpx E,  \kk[\ZZ \M_\tau \cap C_\tau]) \cong \inHom_\kk([\cpx E]_a,  \kk) \otimes_\kk  \kk[\ZZ \M_\tau \cap C_\tau]$.  
Since $[\cpx E]_a$ is an acyclic complex whose 0th cohomology is $\kk$,  $\inHom_R(\cpx E,  \kk[\ZZ \M_\t \cap C_\t])$  
is an  acyclic complex whose 0th cohomology is $\kk[\ZZ \M_\t \cap C_\t]$.  

Anyway, we have $\upl \cpx I_R \cong \inHom_R(\cpx E, \upl I^i_R) \cong  \inHom_R(\cpx E, \cpx D_R) \cong  \cpx D_R$, 
where the middle isomorphism is given by $\Phi(R)$. 
\qed

\medskip

The converse of Theorem~\ref{main} also holds. 

\begin{prop}\label{conv. of main} 
Let $R=\kk[\MM]$ be a toric face ring. If $\upl \cpx I_R$ is quasi-isomorphic to the dualizing complex $\cpx D_R$,  
then $R$ is seminormal. 
\end{prop}

\begin{proof}
Recall that  $\upl \MM:=\{ \upl \M_\s \}_{\s \in \cell}$ forms a monoidal complex, and the toric face ring 
$\upl R=\kk[\upl \MM]$ is the seminormalization of $R$. 
Since $\upl \cpx I_{\upl R}= \upl \cpx I_R$, the proof of the latter half of Theorem~\ref{main} 
also works here. 
\end{proof}

\section{Local cohomologies}

Recall that a monoidal complex $\MM=\{ \M_\s\}_{\s \in \cell}$ is 
a collection of additive submonoids $\M_\s$ of  
lattices $\bL_\s \cong \ZZ^{\dim \s +1}$ for each $\s \in \cell$, and we have an injective homomorphisms 
$\tilde{\i}_{\s, \t}:\bL_\t \to \bL_\s$ for all $\s, \t \in \cell$ with $\s \geq \t$. 
Set
$$
\LL := \colimit_{\s \in \cell}\bL_\s.  
$$
Note that $\LL$ is no longer a group in general.  
Since all $\tilde{\i}_{\s,\t}$ is injective, we can regard $\bL_\s$ as a subset of $\LL$. 
Let $a, b \in \LL$. If there is some $\s \in \cell$ with $a, b \in \bL_\s$, 
we have $a + b \in \bL_\s \subset \LL$. 
If there is no $\s \in \cell$ with $a,b \in \bL_\s$, then $a + b$ does not exist. 
However, any $a \in \LL$ has $-a \in \LL$. 
We can regard that $|\bMM| \subset \LL$, and the structure of $\LL$ defined above and that of $|\bMM|$ are compatible with this injection. 

\begin{dfn}
Let $R := \fring \MM$ be a toric face ring. 
Then $M \in \Mod R$ is said to be $\LL$-graded if the following conditions are satisfied;
\begin{itemize}
  \item[(i)] $M = \bigoplus_{a \in \LL}M_a$ as $\kk$-vector spaces;
  \item[(ii)] $x^a M_b \subset M_{a + b}$ if $a \in \M_\s$ and $b \in \bL_\s$ for some $\s \in \cell$, and $x^a M_b = 0$ otherwise.
\end{itemize}
Let $\Lgr R$ be the category of $\LL$-graded $R$-modules and $R$-homomorphisms $f:M \to N$ with $f(M_a) \subset N_a$ for all $a \in \LL$. 
\end{dfn} 

Clearly, $\Mgr R$ and $\Mbgr R$ are full subcategories of $\Lgr R$. 
Note that $T_\s := \{ \, x^a \mid a \in \M_\s \, \} \subset R$ is a multiplicatively closed subset.
As shown in \cite[Lemma~2.1]{OY}, the localization $T_\s^{-1}R$  is $\LL$-graded.

Well, set 
$$
\Cec^i_R := \bigoplus_{\substack{\s \in \cell \\ \dim \s = i-1}}T_\s^{-1}R
$$
and  define $\partial:\Cec^i_R \to \Cec^{i+1}_R$
by
$$
\partial(x) = \sum_{\substack{\t \ge \s \\ \dim \t = i}} \e(\t,\s) \cdot \iota_{\t, \s}(x)
$$
for $x \in T^{-1}_\s R \subset \Cec_R^i$,
where $\e$ is an incidence function on $\cell$ and $\iota_{\t, \s}$ is a natural map $T^{-1}_\s R \to T^{-1}_\t R$
for $\s \le \t$.
Then $(\Cec^\bullet_R, \partial)$ forms a cochain complex in $\Lgr R$:
$$
0 \longto \Cec^0_R \too \Cec^1_R \too \cdots \too \Cec^d_R \longto 0. 
$$

We set $\m := (x^a \mid 0 \not= a \in \sM)$. This is a maximal ideal of $R$.
The following result has been proved by  Ichim and R\"omer \cite{IR} in the case 
$\MM$ comes from a fan in $\RR^d$, and Okazaki and the present author 
in the general case. (The proofs are essentially the same.)

\begin{prop}[{\cite[Theorem~4.2]{IR}, \cite[Proposition~3.2]{OY}}]\label{Cech}
For any $R$-module $M$, we have 
$$
H^i_\m(M) \cong H^i(\Cec^\bullet_R \otimes_R M),
$$
for all $i$. In particular, $H_\m^i(R)$ is $\LL$-graded. 
\end{prop}

\begin{cor}
Let $\cell$ be a CW complex supporting $R=\kk[\MM]$, and  $X$ 
the underlying topological space of $\cell$. 
Then we have 
$[H_\m^i(R)]_0 \cong \rH^{i-1}(X;\kk)$, 
where  $0$ is the zero element of $\LL$ and $\rH^{i-1}(X;\kk)$ is the $i^{\rm th}$ reduced cohomology 
of $X$ with the coefficients in $\kk$. 
\end{cor}
\begin{proof}
Since $[T_\s^{-1}R]_0=\kk$ for all $\s \in \cell$, the cochain complex 
$[\Cec^\bullet_R]_0$ of $\kk$-vector spaces is isomorphic to 
the reduced cochain complex of $\cell$ with the coefficients in $\kk$. 
Hence the assertion follows from Proposition~\ref{Cech}. 
\end{proof}

For  $M \in \Lgr R$, set $M_{-|\bMM|}:= \bigoplus_{a  \in |\bMM|} M_{-a}$. 
Since $M_{-|\bMM|}$ is not an $R$-module in general, we just regard it as an $\LL$-graded $\kk$-vector space.

\begin{lem}\label{seminormal locla cohomology}
If a toric face ring $R=\kk[\MM]$ is seminormal, then  we have 
$$H_\m^i(R)=[H_\m^i(R)]_{-|\bMM|}$$ 
for all $i$. 
\end{lem}

\begin{proof}
We use the same idea as the proof of Theorem~\ref{Phi}. 
Let $\Sq R$ be the category of squarefree $R$-modules. (See Definition~\ref{Sq}.)

Let $\Lvect \kk$ be the category of $\LL$-graded $\kk$-vector spaces, 
and $(-)_{ -|\bMM|} :\Lgr R \to \Lvect \kk$ the functor which sends $M$ to $M_{-|\bMM|}$. 
We also have the forgetful functor  $\for : \Lgr R \to \Lvect \kk$.

Now, for each $i \in \ZZ$, we define the following two functors from $\Db(\Sq R)$ to $\Lvect \kk$: 
$$\bF_i: \for \circ H^i(- \otimes_R \check{C}^\bullet_R) \qquad \text{and} \qquad 
\bF'_i: [H^i(- \otimes_R \check{C}^\bullet_R)]_{ -|\bMM|}.$$
Since $V_{ -|\bMM|}$ is a subspace of $V \in \Lgr \kk$, we have 
the natural transformation $\Psi_i: \bF'_i \to \bF_i$. 
Since $R$ is seminormal, $\kk[\M_\s]$ is seminormal for all $\s$ by Proposition~\ref{Ngu}. 
Hence $[H_\m^i(\kk[\M_\s])]_{ -|\bMM|}= H_\m^i(\kk[\M_\s])$, 
 in fact, we have  $[H_\m^i(\kk[\M_\s])]_{-C_\s}= H_\m^i(\kk[\M_\s])$ by Theorem~\ref{RRBLR}. 
It means that $\Psi_i(\kk[\M_\s])$ is an isomorphism, and hence 
$\Psi_i$ is a natural isomorphism by the same reason as in the proof of 
Theorem~\ref{Phi}. In particular,  $\Psi_i(R): \bF'_i(R) \to \bF_i(R)$ is an isomorphism. 
Hence $\bF'_i(R)=[H_\m^i(R)]_{ -|\bMM|}$ and $\bF_i(R)=H_\m^i(R)$ are isomorphic. 
\end{proof}

\begin{prop}\label{local cohomology comparison}
Let $R=\kk[\MM]$ be a toric face ring, and $\upl R$ its seminormalization. Then we have 
$$H_\m^i(\upl R) \cong [H_\m^i(R)]_{ -|\bMM|}$$ as $\LL$-graded $\kk$-vector spaces for all $i$.  
\end{prop}

\begin{proof}
It is easy to see that 
$$\{ \, a \in |\bMM| \mid [T_\s^{-1} R ]_{-a} \ne 0 \, \}= \ZZ \M_\s \cap C_\s = 
\{ \, a \in |\bMM| \mid [T_\s^{-1} (\upl R) ]_{-a} \ne 0 \, \}$$
for all $\s \in \cell$. 
Hence we have $(\check{C}_R^\bullet)_{-a}= (\check{C}_{\upl R}^\bullet)_{-a}$ for all $a \in |\bMM|$.
Now the assertion follows from the following computation;
$$[H_\m^i(R)]_{ -|\bMM|} \cong [H^i(\check{C}^\bullet_R)]_{ -|\bMM|} \cong  [H^i(\check{C}^\bullet_{\upl R})]_{ -|\bMM|}
\cong [H_\m^i(\upl R)]_{ -|\bMM|} \cong H_\m^i(\upl R).$$
Here the second ``$\cong$" follows from the fact stated above, and the last one 
is Lemma~\ref{seminormal locla cohomology}. 
\end{proof}

\begin{rem}
In some sense, Proposition~\ref{local cohomology comparison} generalizes and refines 
the results and the problem in \S 4 of Nguyen~\cite{Ngu}  (especially, \cite[Theorem~4.3]{Ngu}).  
However, the toric face rings in \cite{Ngu} are assumed to have nice multigradings, while the ``$\LL$-grading" 
of our $\kk[\MM]$ is not the grading in the usual sense. 
%So our result is not a direct generalization of those in \cite{Ngu}. 
\end{rem}

\begin{cor}\label{R and uplR}
Let $R=\kk[\MM]$ be a toric face ring, and $\upl R$ its seminormalization. 
If $R$ is Cohen-Macaulay, then so is $\upl R$. 
\end{cor}

\begin{proof}
We prove the contrapositive: if $\upl R$ is not Cohen-Macaulay, then neither is $R$. 
Assume that $\upl R$ is not Cohen-Macaulay. Then there is some $0 \le i < \dim R$ with 
$H^{-i}(\upl \cpx I_{\upl R}) \ne 0$. For $a \in |\bMM|$, the cochain complex $[\upl \cpx I_{\upl R}]_a$ 
of $\kk$-vector spaces is isomorphic to the $\kk$-dual of $[\cpx \Cec_{\upl R}]_{-a}$. 
Hence it follows that $H_\m^i(\upl R) \ne 0$. By Proposition~\ref{local cohomology comparison}, we have 
$H_\m^i(R) \ne 0$, and hence the localization $R_\m$ is not Cohen-Macaulay. 
\end{proof}

\begin{prop}\label{R and oR}
For a monoidal complex $\MM= \{ \M_\s\}_{\s \in \cell}$, set $\bMM := \{ \bL_\s \cap C_\s \}_{\s \in \cell}$ as before. 
Let $R:=\kk[\MM]$ and $\oR:=\kk[\bMM]$ be their toric face rings. 
If $R$ is Cohen-Macaulay, then so is $\oR$.  
Moreover, $H_\m^i(\oR) \ne 0$ implies $H_\m^i(R) \ne 0$. 
\end{prop}

\begin{lem}\label{final lemma}
 With the same notation as in Proposition~\ref{R and oR}, 
%Assume that $R$ is seminormal. 
 $H^i(\cpx D_{\oR})\ne 0$ implies $H^i(\upl \cpx I_R)\ne 0$. 
\end{lem}

\begin{proof}
Recall that $\cpx D_{\oR} \cong \cpx I_{\oR}$.   If $H^i(\cpx D_{\oR})( \cong 
H^i(\cpx I_{\oR}) ) \ne 0$, then there is   $a \in |\bMM|$ with 
$[H^i(\cpx I_{\oR})]_a \ne 0$.  
Set $\s:=\supp(a)$ (i.e., $a \in \tM_\s \cap \int(C_\s)$). 
Since $H^i(\cpx I_{\oR})$ is a squarefree $\oR$-module, we have 
 $[H^i(\cpx I_{\oR})]_a \cong  [H^i(\cpx I_{\oR})]_b$ for all $b \in  |\bMM|$ with $\supp (b)=\s$. 

For $b \in \M_\s$ with $\supp (b)=\s$, 
we have $b \in \M_\tau$ for all $\tau \in \cell$ with $\tau \ge \s$. 
In this case, regarding $b \in  \sM \subset |\bMM|$, we have $[\upl \cpx I_R]_b = [\cpx I_{\oR}]_b$ as cochain complexes of $\kk$-vector spaces, and   
hence $[H^i(\upl \cpx I_R)]_b \cong [H^i(\cpx I_{\oR})]_b \ne 0$. 
\end{proof}

\noindent{\it The proof of Proposition~\ref{R and oR}.} 
By Corollary~\ref{R and uplR}, we may assume that $R$ is seminormal. 
Then $\upl \cpx I_R \cong \cpx D_R$ by Theorem~\ref{main}, 
and the assertion  easily follows from  Lemma~\ref{final lemma}.     
\qed

\medskip

Let $R=\kk[\MM]$ be a general toric face ring, $\upl R=\kk[\upl \MM]$ its seminormalization, and $\oR=\kk[\bMM]$.  
Proposition~\ref{R and oR} and Corollary~\ref{R and uplR} state that  
$$\text{$R$ is Cohen-Macaulay} \Longrightarrow \text{$\upl R$ is Cohen-Macaulay} 
\Longrightarrow \text{$\oR$ is Cohen-Macaulay}.$$
By a  result of Caijun~\cite{Cj} (see also \cite{OY}), the Cohen-Macaulay property of $\oR$ is a topological property of the 
underlying space $X$ of $\cell$, while it may depend on $\chara(\kk)$.  

\section*{Acknowledgements}
I am grateful to Professor Yuji Yoshino for giving me valuable comments in the very early stage of this study 
(even before the publication of \cite{OY}). I also thank Professor  Ryota Okazaki for stimulating discussion.

%
%
%---body---
%

%
%
%=======================================================
%
%

\begin{thebibliography}{99}
\bibitem{BG02}
W. Bruns, J. Gubeladze, 
Polyhedral algebras, arrangements of toric varieties, and their groups,  
in: Computational commutative algebra and combinatorics, 
Adv. Stud. Pure Math., vol. 33, 2001, pp. 1--51.

\bibitem{BG}
W. Bruns and J. Gubeladze, Polytopes, rings, and $K$-theory, Springer, 2009. %preliminary version. Available at
%\url{http://www.math.uos.de/staff/phpages/brunsw/preprints.htm}

\bibitem{BH} 
W. Bruns and J. Herzog,  Cohen-Macaulay rings, revised edition,  
Cambridge University Press, 1998.

\bibitem{BKR}
W. Bruns, R. Koch, and T. R\"omer,
Gr\"obner bases and Betti numbers of monoidal complexes, 
Michigan Math. J.  {\bf 57} (2008), 71-91. 

\bibitem{BLR} W. Bruns, P. Li, T. R\"omer, On seminormal monoid rings, J. Algebra 302 (2006) 361--386. 

\bibitem{Cj} Z. Caijun, 
Cohen-Macaulay section rings,   
Trans. Amer. Math. Soc. 349 (1997), 4659--4667. 


\bibitem{GW}
S. Goto and K. Watanabe, 
On graded rings. II. (${Z}\sp{n}$-graded rings). 
Tokyo J. Math. {\bf 1} (1978), 237--261.  


\bibitem{Ha} R. Hartshorne,
Residues and duality,
Lect. Notes in Math. {\bf 20}, Springer, 1966. 

\bibitem{IR} B. Ichim and T. R\"omer,
On toric face rings, J. Pure Appl. Algebra {\bf 210} (2007), 249--266.


\bibitem{I} M.-N. Ishida, 
The local cohomology groups of an affine semigroup ring, 
in: Algebraic Geometry and Commutative Algebra, vol. I, Kinokuniya, Tokyo, 1988, pp. 141--153. 

\bibitem{MS} E. Miller and B. Sturmfels, 
Combinatorial Commutative Algebra, 
Grad. Texts in Math., Vol. 227,  Springer, 2004. 


\bibitem{Ngu} D.H. Nguyen,  Seminormality and local cohomology of toric face rings,  J. Algebra {\bf 371} (2012) 536--553. 

\bibitem{OY} R. Okazaki and K. Yanagawa, Dualizing complex of a toric face ring,  
Nagoya Math. J {\bf 196} (2009), 87--116. 

\bibitem{RR} L. Reid, L.G. Roberts, Monomial subrings in arbitrary dimension, J. Algebra {\bf 236}  (2001) 703--730. 

\bibitem{SS}
U. Schafer and  P. Schenzel,  Dualizing complexes of affine semigroup rings. Trans. Am. Math. Soc. {\bf 322},  (1990) 561--582.  


\bibitem{St} R.P. Stanley, 
Generalized $H$-vectors, intersection cohomology of toric varieties,
and related results: 
in: Commutative algebra and combinatorics,
Adv. Stud. Pure Math., {\bf 11}, 1987, 187--213. 

\bibitem{Sw} 
R.G. Swan, On seminormality, J. Algebra {\bf 67} (1980) 210--229.
\end{thebibliography}
\end{document}